%% file: main.tex
\iffalse
\documentclass[letterpaper, 10 pt, conference]{ieeeconf}
\iffalse
\documentclass[letterpaper, 12pt,conference]{ieeeconf}
\onecolumn
\fi
\IEEEoverridecommandlockouts                         
\else
\documentclass[11pt,letterpaper,oneside,reqno]{article}
\fi 

\usepackage{environ}
\input{mystyle}
\input{notation}

\title{\LARGE \bf
Convex Nonparametric Formulation for Identification of Gradient Flows}

\author{Mohammad~Khosravi\thanks{Corresponding author}~ and~ Roy~S.~Smith%
	\thanks{This research project is part of the Swiss Competence Center for Energy
		Research SCCER FEEB\&D of the Swiss Innovation Agency Innosuisse.}
	\ \thanks{The authors are with Automatic Control Lab,  ETH Zurich, Switzerland\\
	 {\tt\small  \{khosravm,rsmith\}@control.ee.ethz.ch}}
}

\begin{document}
\maketitle
\thispagestyle{empty}
\pagestyle{empty}
\input{sec_00_abstr}
\input{sec_01_intro}
\input{sec_02_a_notation}
\input{sec_03_problem}	
\input{sec_04_CVX_Energy}	
\input{sec_05_NonCVX_Energy}	
\input{sec_06_NumericalResults}	
\input{sec_07_Concl}	
\bibliographystyle{IEEEtran}
\bibliography{mybib}
\end{document}

%% file: mystyle.tex
\usepackage[noadjust]{cite}
\usepackage[colorlinks=true]{hyperref}
\usepackage{amssymb, amsmath, amsfonts}
\usepackage{authblk, dsfont, mathrsfs, color, bm,bbm, url}
\usepackage{graphicx, epstopdf, pgfplots, transparent} 
\usepackage{algorithm, algorithmicx, multirow, titlecaps}
\usepackage{threeparttable}
\usepackage[T1]{fontenc}
\usepackage{subcaption}

\usepackage{tikzsymbols}
\DeclareMathAlphabet{\mathbbold}{U}{bbold}{m}{n}

\usepackage[english]{babel}
\usepackage[noend]{algpseudocode} 

\newcommand{\argmin}{{\mathrm{argmin}}}

\newcommand*{\argminOp}{\operatornamewithlimits{argmin}\limits}

\newcommand*{\minOp}{\operatornamewithlimits{min}\limits}
\newcommand*{\maxOp}{\operatornamewithlimits{max}\limits}
\newcommand*{\sumOp}{\operatornamewithlimits{\sum}\limits}

\newcommand{\tr}{{\mathsf{T}}}

\newcommand{\zero}{\mathbf{0}}
\newcommand{\one}{\mathbf{1}}
\iftrue
\newcommand{\eye}{\mathbb{I}}
\else
\newcommand{\eye}{\mathbf{I}}
\fi

\newcommand{\vc}[1]{{ \mathrm{#1} }}
\newcommand{\mx}[1]{{ \mathrm{#1} }}
\newcommand{\conv}{\mathrm{conv}} 
\newcommand{\diag}{\mathrm{diag}}
\newcommand{\Diag}{\mathrm{Diag}}

\newcommand{\inner}[2]{{ \langle {#1,#2} \rangle}}

\newcommand{\Dscr}{{\mathscr{D}}}

\newcommand{\Lscr}{{\mathscr{L}}}

\newcommand{\Acal}{{\mathcal{A}}}

\newcommand{\Ccal}{{\mathcal{C}}}

\newcommand{\Ical}{{\mathcal{I}}}
\newcommand{\Jcal}{{\mathcal{J}}}
\newcommand{\Kcal}{{\mathcal{K}}}
\newcommand{\Lcal}{{\mathcal{L}}}

\newcommand{\Rcal}{{\mathcal{R}}}

\newcommand{\Ucal}{{\mathcal{U}}}

\newcommand{\Ycal}{{\mathcal{Y}}}
\newcommand{\Zcal}{{\mathcal{Z}}}

\newcommand{\Nbb}{{\mathbb{N}}}

\newcommand{\Rbb}{{\mathbb{R}}}
\newcommand{\Sbb}{{\mathbb{S}}}

\newcommand{\RR}{\mathbb{R}}

\usepackage{amsthm}

\theoremstyle{plain}
\newtheorem{theorem}{Theorem}

\newtheorem{corollary}[theorem]{Corollary}

\newtheorem{remark}{Remark}

\newtheorem*{problem*}{Problem}

\let\labelindent\relax
\usepackage{enumitem}
\newlist{todolist}{itemize}{2}
\setlist[todolist]{label=$\square$}
\usepackage{pifont}
\iftrue 
\usepackage{xcolor}
\hypersetup{
	colorlinks,
	linkcolor={blue!55!black},
	citecolor={red!55!black},
	urlcolor={blue!80!black}
}
\usepackage[letterpaper, total={7in, 9in}]{geometry}
\fi

%% file: notation.tex
\newcommand{\nS}{{n_{\text{s}}}}
\newcommand{\nT}{{n_{\text{T}}}}

\newcommand{\xib}{\boldsymbol{\xi}}
\newcommand{\thetab}{\boldsymbol{\theta}}

\newcommand{\xibl}{\boldsymbol{\xi}_{\lambda}}
\newcommand{\thetabl}{\boldsymbol{\theta}_{\lambda}}
\newcommand{\varphil}{{\varphi}_{\lambda}}
\newcommand{\varphilt}{{\varphi}_{\lambda,\tau}}

\newcommand{\etabl}{\boldsymbol{\eta}_{\lambda}}

\newcommand{\varphione}{{\varphi}^{(1)}}
\newcommand{\varphitwo}{{\varphi}^{(2)}}
\newcommand{\xione}{{\xi}^{(1)}}
\newcommand{\xitwo}{{\xi}^{(2)}}

\newcommand{\xibone}{{\xib}^{(1)}}
\newcommand{\xibtwo}{{\xib}^{(2)}}
\newcommand{\thetaone}{{\theta}^{(1)}}
\newcommand{\thetatwo}{{\theta}^{(2)}}
\newcommand{\thetabone}{\boldsymbol{\theta}^{(1)}}
\newcommand{\thetabtwo}{\boldsymbol{\theta}^{(2)}}

\newcommand{\thetablone}{\boldsymbol{\theta}^{(1)}_{\lambda}}
\newcommand{\thetabltwo}{\boldsymbol{\theta}^{(2)}_{\lambda}}
\newcommand{\xiblone}{{\xib}^{(1)}_{\lambda}}
\newcommand{\xibltwo}{{\xib}^{(2)}_{\lambda}}
\newcommand{\varphilone}{{\varphi}_{\lambda}^{(1)}}
\newcommand{\varphiltwo}{{\varphi}_{\lambda}^{(2)}}
\newcommand{\varphiltone}{{\varphi}_{\lambda,\tau}^{(1)}}
\newcommand{\varphilttwo}{{\varphi}_{\lambda,\tau}^{(2)}}

\newcommand{\Rbbp}{\mathbb{R}_{
		\scalebox{0.5}{\text{$\ge 0$}}
}}
\newcommand{\Zbbp}{\mathbb{Z}_{
		\scalebox{0.5}{\text{$\ge 0$}}
}}


%% file: sec_00_abstr.tex
\begin{abstract}
In this paper, we develop a nonparametric system identification method for the nonlinear gradient-flow dynamics. In these systems, the vector field is the gradient field of a potential energy function. This fundamental fact about the dynamics of system plays the role of a structural prior knowledge as well as a constraint in the proposed identification method. While the nature of the identification problem is an estimation in the space of functions, we derive an equivalent finite dimensional formulation, which is a convex optimization in form of a quadratic program. This gives scalability of the problem and provides the opportunity for utilizing recently developed large-scale optimization solvers. The central idea in the proposed method is representing the energy function as a difference of two convex functions and estimating these convex functions jointly. Based on necessary and sufficient conditions for function convexity, the identification problem is formulated, and then, the existence, uniqueness and smoothness of the solution is addressed. We also illustrate the method numerically for a demonstrative example.	
\end{abstract}

%% file: sec_01_intro.tex
\section{Introduction}\label{sec:int}
Nonlinear dynamics are ubiquitous in nature and widely used for modeling various phenomena in physics, chemistry, biology and other fields of science and engineering \cite{strogatz2001nonlinear, chen2018neural}. These models are either derived from first principles or by means of fitting and estimation methods. The latter employ techniques in optimization, statistical learning theory and system identification for deriving the model from the available measurement data. Meanwhile, in many cases modeling is beyond fitting a nonlinear dynamics to the observation data. We may additionally need to incorporate specific properties into the model, which are inherent in the nature of system. These properties includes stability, passivity, positivity or other possible characteristics of the system.  For the case of linear dynamics, many of these properties are already taken into account \cite{pillonetto2014kernel,goethals2003identification,khosravi2019positive}.  

For the nonlinear systems, a similar line of research has received extensive attention in the past decade \cite{sattar2020non,singh2019learning,kaiser2018sparse}. For example,  
identifying stabilizable non-autonomous dynamics is investigated in \cite{sattar2020non,singh2019learning}, and learning the dynamics subject to sparsity of the dynamic modes of system is discussed in \cite{kaiser2018sparse}.  
For the purpose of imitation learning, the dynamics modeled in \cite{calinon2010learning} based on Gaussian mixtures and hidden Markov models. 
Meanwhile, a similar approach is employed in \cite{khansari2011learning} with a global stability guarantee.
In \cite{khansari2017learning},  a convex quadratic potential energy as well as a linear dissipative field are considered with respect to each data point, and  the dynamics are modeled as a functional weighted sum of the gradient forces and the dissipative fields.
Also, in \cite{ijspeert2013dynamical}, the dynamics are modeled as weakly nonlinear differential equations which have a linear part for capturing the baseline behavior together with more complicated coupling dynamics for considering more complex phenomena.

An interesting class of nonlinear dynamics are {\em  gradient-flows}, also known as {\em curl-free vector fields}.
In physics, these vector fields are  called {\em conservative forces} with classical examples being electric and gravitational forces.
The gradient flows are defined as the 
negative of gradient of a potential energy function.
This property of gradient flow can be used as a structural prior knowledge as well as a constraint in the identification problem. 
In \cite{sindhwani2018learning}, a learning method is introduced based on the notion of vector-valued kernels which might be suitable for learning gradient flow of convex energy functions. However, the arguments in \cite{sindhwani2018learning} do not provide concise theoretical guarantees. 

Inspired by recent progresses in shape-constrained and convex regression \cite{mazumder2019computational}, we introduce a  nonparametric identification method for the gradient-flow dynamics.
The introduced identification problem is originally an estimation problem in the space of functions, i.e., it is a minimization of fitting or prediction error over the hypothesis space of  convex functions. 
Meanwhile, we derive an equivalent finite dimensional convex optimization problem. 
For the sake of more transparent discussion and ease of notation, in Section \ref{sec:cvx_energy}, first analyzes the case of convex energy functions. The results are then extended to the cases where the knowledge of the equilibrium is available, the energy functions are strongly convex as well as (strongly) concave, and subsequently the case of general energy functions in Section \ref{sec:noncvx_energy}.
Finally, in Section \ref{sec:numerics}, the method is numerically illustrated on a demonstrative
example.	

%% file: sec_02_a_notation.tex
\section{Notations and Preliminaries}
The set of natural numbers, the set of non-negative integers, the set of real numbers, $n$-dimensional Euclidean space and the space of $n$ by $m$ real matrices are denoted by $\Nbb$, $\Zbbp$,  $\Rbb$, $\Rbb^n$, and $\Rbb^{n\times m}$ respectively.
The identity matrix and zero vector in the Euclidean space are denoted by $\eye$ and $\zero$ respectively.
The set of symmetric positive definite matrices in $\Rbb^{n\times n}$ is denoted by $\Sbb_{++}^n$.
For any pair of symmetric matrices $\mx{X},\mx{Y}\in\RR^{n\times n}$, we write $\mx{X}\succeq \mx{Y}$ if $\mx{X}-\mx{Y}\in\Sbb_{++}^n$.
For a set $\Acal\subseteq\Rbb^n$, the convex hull of $\Acal$ is denoted by $\conv(\Acal)$. 
The Euclidean norm and the inner product on $\Rbb^n$ are respectively denoted by $\|\cdot\|$ and $\inner{\cdot}{\cdot}$.
For a function $f$, $\nabla\! f$ and $\nabla^2\!f$ are the gradient and Hessian of $f$ respectively.
For a convex function $\varphi:\Ucal\subseteq\Rbb^n\to\Rbb$, the subgradient or subderivative of $f$ at $\vc{x}\in\Ucal$ is denoted by $\partial f(\vc{x})$ and defined as the set of vectors $\xi\in\Rbb^n$ satisfying the inequality 
$\varphi(\vc{y})-\varphi(\vc{x})\ge \inner{\xi}{\vc{y}-\vc{x}},\quad \forall \vc{y}\in\Ucal$. 
Let $\Ycal$ be a set and $\Ccal$ be a subset of $\Ycal$. The indicator function of $\Ccal$, denoted by $\Ical_{\Ccal}$, is defined as
$\Ical_{\Ccal}(y) = 0$, if $y\in\Ccal$ and  $\Ical_{\Ccal}(y) = \infty$, otherwise.

%% file: sec_03_problem.tex
\section{Problem Statement}
Let $\Ucal$ be a simply-connected open subset of $\Rbb^n$ and $\varphi:\Ucal\to \Rbb$ be an \underline{\em unknown} function.
We call $\varphi$ the {\em potential energy} function or simply, {\em energy} function.
A conservative vector field corresponding to $\varphi$ is induced over the space, and the corresponding dynamics are defined as 
\begin{equation}\label{eqn:xdot=f(x)=-Dphi(x)}
\begin{array}{c}
\dot{\vc{x}} = f(\vc{x}):=-\nabla\! \varphi(\vc{x}).
\end{array}
\end{equation}
Starting from initial condition $\vc{x}_0\in\Ucal$ at time $t=0$, the vector field generates a trajectory which is denoted here by $\vc{x}(\cdot;\vc{x}_0)$.
Consider a set of initial points $\vc{x}_0^1,\ldots,\vc{x}_0^\nT$ and corresponding trajectories.
For any $i=1,\ldots,\nT$, let trajectory $\vc{x}(\cdot;\vc{x}_0^i)$ be sampled at time instants $0\le t^i_1<t^i_2<\cdots<t^i_{n_i}$, where $n_i\in\Nbb$. Let $\vc{x}_k^i$ denote $\vc{x}(t_k^i;\vc{x}_0^i)$ for $1\le k\le n_i$.
The time derivative of $\vc{x}(t;\vc{x}^i_0)$ at each sampling time instant can then be estimated by simply utilizing a nonlinear regression method and subsequently obtaining the derivatives numerically or analytically. Various other techniques, e.g. see \cite{wang2019robust} and the references therein are also available in the literature to estimate the derivative of the trajectory. Let these estimations be denoted by $\vc{y}_k^i$, for $1\le k\le n_i$.
Note that $\vc{y}_k^i$ is approximately equal to $f(\vc{x}_k^i)$.
Based on these samples and estimations, we can introduce a set of {\em data}, denoted by $\Dscr$, which contains data pairs $(\vc{x}_k^i,\vc{y}_k^i)$.
More precisely, $\Dscr$ is defined as $\{(\vc{x}_j,\vc{y}_j)\ |\ 1\le j\le \nS\}$, where $\nS:=\sum_{1\le i\le\nT}n_i$ and, for simplicity of notation, the superscripts are dropped. 
\begin{problem*}
Given the set of data $\Dscr$, estimate the unknown vector field $f$ in \eqref{eqn:xdot=f(x)=-Dphi(x)}.
\end{problem*}
\begin{remark}
This problem is a nonlinear system identification where {\em structural prior knowledge} is provided in form of \eqref{eqn:xdot=f(x)=-Dphi(x)}.
\end{remark}
\begin{remark}
The dynamics in \eqref{eqn:xdot=f(x)=-Dphi(x)} can be extended to the case of {\em differential inclusions} \cite{aubin2012differential}. More precisely, one may consider $\dot{\vc{x}}\in\pm\partial\varphi(\vc{x})$ or $\dot{\vc{x}}\in-(\partial\varphi_1(\vc{x})-\partial\varphi_2(\vc{x}))$, where $\varphi$, $\varphi_1$ and $\varphi_2$ are convex functions and $\partial$ is the sub-derivative operator.
\end{remark}
\begin{remark}
The dynamics in \eqref{eqn:xdot=f(x)=-Dphi(x)} can be extended to the case of {\em differential inclusions} \cite{aubin2012differential}. More precisely, one may consider $\dot{\vc{x}}\in\pm\partial\varphi(\vc{x})$ or $\dot{\vc{x}}\in-(\partial\varphi_1(\vc{x})-\partial\varphi_2(\vc{x}))$, where $\varphi$, $\varphi_1$ and $\varphi_2$ are convex functions and $\partial$ is the sub-derivative operator.
\end{remark}

%% file: sec_04_CVX_Energy.tex
\section{Convex Energy Functions}\label{sec:cvx_energy}
In this section, we consider the case where the energy function is convex. In the following, we relax the differentiablity assumption of energy function for the sake of generality.

Let $\Phi$ denote the set of convex functions defined over $\Rbb^n$.
Considering the data $\Dscr$,
we define the loss function for the estimation problem, denoted by $\Lcal_{\Phi,\Dscr}$, as the sum of squared errors. 
In other words, $\Lcal_{\Phi,\Dscr}:\Phi\times\Rbb^{\nS n}\to \Rbb$ is a function such that for any given convex function $\varphi \in \Phi$ and vectors $\xi_1,\ldots,\xi_{\nS}$, the value of $\Lscr_{\Phi,\Dscr}(\varphi,\xib)$ is defined as 
\begin{equation}\label{eqn:loss_phi}
\Lcal_{\Phi,\Dscr}(\varphi,\xib) := \sum_{i=1}^{\nS}  \|\vc{y}_i+\xi_i\|^2 + \sum_{i=1}^{\nS} \Ical_{ \partial \varphi(\vc{x}_i)}(\xi_i), 
\end{equation}
where $\xib \in\Rbb^{\nS n}$ is a column vector defined as $\xib :=
\begin{bmatrix}
\xi_1^\tr & \ldots & \xi_{\nS}^\tr \end{bmatrix}^\tr$.
Note that for any $\varphi\in\Phi$, we have that $\partial \varphi(\vc{x}) \ne \emptyset$, for any $\vc{x}$, and also, $\partial \varphi(\vc{x}) = \{\nabla\!\varphi(\vc{x})\}$, when $\varphi$ is differentiable at $\vc{x}$.
Accordingly, the estimation problem is naturally defined as
\begin{equation}\label{eqn:opt_phi}
\begin{array}{cl}
\minOp_{
\xi_1,\ldots,\xi_{\nS}\in \Rbb^n, \varphi\in\Phi} &  
\sumOp_{i=1}^{\nS} \|\vc{y}_i+\xi_i\|^2,\\
\textrm{s.t.} & \xi_i\in\partial\varphi(\vc{x}_i), \quad\forall i=1,\ldots,\nS. 
\end{array}
\end{equation}
Note that optimization problem \eqref{eqn:opt_phi} is over the set $\Phi$, a cone in the space of functions which is an infinite-dimensional space. We investigate this problem and introduce a tractable approach for obtaining a solution for \eqref{eqn:opt_phi}.
\subsection{Towards Finite-Dimensional Formulation}
For any convex function $\varphi$, the following holds \cite{nesterov2018lectures} 
\begin{equation}\label{eqn:cvx_var_ineq_subgrad}
\varphi(\vc{y})-\varphi(\vc{x})\ge \inner{\xi}{\vc{y}-\vc{x}},\quad \forall \vc{x},\vc{y}\in\Rbb^n,\ \forall \xi\in\partial\varphi(\vc{x}).
\end{equation}
Motivated by this property of convex functions, we introduce the following optimization problem
\begin{equation}\label{eqn:opt_thetaxi}
\begin{array}{cl}
\minOp_{\substack{
		\xi_1,\ldots,\xi_{\nS}\in \Rbb^n\\
		\!\!\!\!\theta_1,\ldots,\theta_{\nS} \in\Rbb}} &  
	\sumOp_{i=1}^{\nS} \|\vc{y}_i+\xi_i\|^2,
\\
\textrm{s.t.} & \theta_j-\theta_i\ge \inner{\xi_i}{\vc{x}_j-\vc{x}_i}, \qquad \forall i,j\in\{1,\ldots,\nS\}. 
\end{array}
\end{equation}
Define the vector $\thetab\in\Rbb^{\nS}$ as
$ \thetab :=
\begin{bmatrix}
\theta_1 & \ldots&\theta_{\nS} \end{bmatrix}^\tr$.
Let $\Kcal$ be the feasible set in \eqref{eqn:opt_thetaxi},  
\begin{equation}
\begin{split}
\Kcal := \bigg{\{}(\thetab,\xib)\in \Rbb^{\nS} \times \Rbb^{\nS n}\ \bigg{|}\ \theta_j-\theta_i & \ge  \inner{\xi_i}{\vc{x}_j-\vc{x}_i}, 
\quad
\forall i,j\!=\!1,\ldots,\nS\bigg{\}}.
\end{split}
\end{equation}
Considering the optimization problem \eqref{eqn:opt_thetaxi}, one can define a loss function $\Lcal_{\Kcal,\Dscr}:\Rbb^{\nS} \times \Rbb^{\nS n}\to \bar{\Rbb}$ as
\begin{equation}
\Lcal_{\Kcal,\Dscr}(\thetab,\xib) := \sum_{i=1}^{\nS} \|\vc{y}_i+\xi_i\|^2 + \Ical_{\Kcal}(\thetab,\xib).
\end{equation} 
The next theorem presents the connection between optimization problems \eqref{eqn:opt_phi} and \eqref{eqn:opt_thetaxi} as well as the corresponding loss functions. 
\begin{theorem}\label{thm:correspondence_0}
{\rm (i)} Let $(\varphi,\xib)$ be a solution of \eqref{eqn:opt_phi}. Then, \eqref{eqn:opt_thetaxi} has a solution $(\thetab,\xib)$ such that 
\begin{equation}
\Lcal_{\Kcal,\Dscr}(\thetab,\xib) = \Lcal_{\Phi,\Dscr}(\varphi,\xib).
\end{equation} 
{\rm (ii)} Conversely, if $(\thetab,\xib)$ is a solution of \eqref{eqn:opt_thetaxi}, then \eqref{eqn:opt_phi} has a solution $(\varphi,\xib)$ such that 
\begin{equation}
\Lcal_{\Phi,\Dscr}(\varphi,\xib) = \Lcal_{\Kcal,\Dscr}(\thetab,\xib),
\end{equation} 
and $\xi_i\in \partial\varphi(\vc{x}_i)$, for all $i=1,\ldots,\nS$. 
\end{theorem}
\begin{proof}
{\it Proof of {\rm (i)}:} For any $i\in\{1,\ldots,\nS\}$, define $\theta_i=\varphi(\vc{x}_i)$. From \eqref{eqn:cvx_var_ineq_subgrad}, one can easily see that $(\thetab,\xib)\in\Kcal$, and subsequently, $\Ical_{\Kcal}(\thetab,\xib) =0 $. Therefore, we have
\begin{equation}\label{eqn:thm_correspondence_eq1}
\Lcal_{\Kcal,\Dscr}(\thetab,\xib) =
\sumOp_{i=1}^{\nS} \|\vc{y}_i+\xi_i\|^2
= 
\Lcal_{\Phi,\Dscr}(\varphi,\xib).
\end{equation}
Let \eqref{eqn:opt_thetaxi} have a feasible point $(\tilde\thetab,\tilde\xib)$ such that $\Lcal_{\Kcal,\Dscr}(\tilde\thetab,\tilde\xib) <\Lcal_{\Kcal,\Dscr}(\thetab,\xib)$.
Therefore, we have
\begin{equation}\label{eqn:thm_correspondence_eq2}
\sum_{i=1}^{\nS} \|\vc{y}_i+\tilde\xi_i\|^2 
\!=\! 
\Lcal_{\Kcal,\Dscr}(\tilde\thetab,\tilde\xib) 
\!<\!
\Lcal_{\Kcal,\Dscr}(\thetab,\xib)
\!=\!
\sum_{i=1}^{\nS} \|\vc{y}_i+\xi_i\|^2.
\end{equation}
Let function $\tilde{\varphi}:\Rbb^n\to\Rbb$ be defined as
\begin{equation}\label{eqn:checkphi}
\begin{array}{c}
\tilde{\varphi}(\vc{x}) : = \max_{i=1,\ldots,\nS}\inner{\tilde\xi_i}{\vc{x}-\vc{x}_i}+\tilde\theta_i, \quad \forall \vc{x}\in\Rbb^n.
\end{array}
\end{equation}
One can easily see that $\tilde{\varphi}$ is a convex function, i.e., $\tilde{\varphi}\in\Phi$.
Define set-valued map $\tilde{I}:\Rbb^n\rightrightarrows \{1,\ldots,\nS\}$ as 
\begin{equation}\label{eqn:Imap}
\tilde{I}(\vc{x})=\bigg{\{}i\in\{1,\ldots,\nS\} \ \bigg{|}\ \tilde{\varphi}(\vc{x}) = \inner{\tilde\xi_i}{\vc{x}-\vc{x}_i} +
\tilde\theta_i, \bigg{\}}.
\end{equation} 
For any $\vc{x}\in\Rbb^n$, we know that
\cite{nesterov2018lectures} 
\begin{equation}\label{eqn:subgrad_checkphi}
\partial \tilde{\varphi}(\vc{x})= \conv\Big{\{}\tilde\xi_i\ \Big{|}\ i\in \tilde{I}(\vc{x})\Big{\}}.
\end{equation}
Since $(\tilde\thetab,\tilde\xib)\in\Kcal$,  for any $i=1,\ldots,\nS$, we have
\begin{equation}\label{eqn:theta_i_theta_j}
\tilde\theta_i
\ge 
\inner{\tilde\xi_j}{\vc{x}_i-\vc{x}_j} +
\tilde\theta_j,\qquad \forall j=1,\ldots,\nS.
\end{equation}
Therefore, from \eqref{eqn:checkphi} and \eqref{eqn:theta_i_theta_j}, one can see
\begin{equation}
\tilde{\varphi}(\vc{x}_i)\ge
\inner{\tilde\xi_i}{\vc{x}_i-\vc{x}_i} +
\tilde\theta_i= 
\tilde\theta_i \ge \max_{j=1,\ldots,\nS}\inner{\tilde\xi_j}{\vc{x}_i-\vc{x}_j}+\tilde\theta_j = \tilde{\varphi}(\vc{x}_i).
\end{equation}
Subsequently, due to \eqref{eqn:subgrad_checkphi}, one can see $\tilde{\xi_i}\in \partial \tilde{\varphi}(\vc{x}_i)$. Therefore, $(\tilde{\varphi},\tilde{\xib})$ is a feasible point for \eqref{eqn:opt_phi}. 
Subsequently, due to \eqref{eqn:thm_correspondence_eq1} and \eqref{eqn:thm_correspondence_eq2}, we have
\begin{equation}
\Lcal_{\Phi,\Dscr}(\tilde{\varphi},\tilde{\xib})= \sum_{i=1}^{\nS}  \|\vc{y}_i + \tilde{\xi}_i\|^2 
= 
\Lcal_{\Kcal,\Dscr}(\tilde\thetab,\tilde\xib)
<
\Lcal_{\Kcal,\Dscr}(\thetab,\xib) 
= \Lcal_{\Phi,\Dscr}(\varphi,\xib),
\end{equation} 
which is a contradiction and we have $\Lcal_{\Kcal,\Dscr}(\tilde\thetab,\tilde\xib) \le \Lcal_{\Kcal,\Dscr}(\thetab,\xib)$.
Therefore $(\thetab,\xib)$ is a solution of \eqref{eqn:opt_thetaxi} and the proof of part {\rm{(i)}} is concluded.\\
{\it Proof of {\rm (ii)}:}  
Let $(\thetab,\xib)$ be a solution of \eqref{eqn:opt_thetaxi}. Define $\varphi:\Rbb^n\to \Rbb$ as
\begin{equation}\label{eqn:thm_varphi_max_0}
\varphi(\vc{x}) : = \maxOp_{i=1,\ldots,\nS}\inner{\xi_i}{\vc{x}-\vc{x}_i}+\theta_i.
\end{equation}
Note that $\varphi$ is a convex function, i.e. $\varphi\in\Phi$. 
Define set-valued map $I:\Rbb^n\rightrightarrows \{1,\ldots,\nS\}$ similar to \eqref{eqn:Imap}. 
For any $\vc{x}\in\Rbb^n$, we have 
$\partial \varphi(\vc{x})
= 
\conv\{\xi_i\ |\ i\in I(\vc{x})\}$.
Since $(\thetab,\xib)\in\Kcal$, based on a similar argument to the proof of part {\rm{(ii)}}, we have that $\xi_i\in\partial \varphi(\vc{x}_i)$, for any $i=1,\ldots,\nS$. Subsequently, we have 
\begin{equation}\label{eqn:thm_correspondence_eq3}
\Lcal_{\Phi,\Dscr}(\varphi,\xib) = 
\sum_{i=1}^{\nS} \|\vc{y}_i+\xi_i\|^2
= \Lcal_{\Kcal,\Dscr}(\thetab,\xib).
\end{equation}
Now, let \eqref{eqn:opt_phi} have a feasible point $(\tilde{\varphi},\tilde{\xib})$ such that 
$\Lcal_{\Phi,\Dscr}(\tilde{\varphi},\tilde{\xib}) <\Lcal_{\Phi,\Dscr}(\varphi,\xib)$. 
For any $i=1,\ldots,\nS$, define $\tilde{\theta}_i = \tilde{\varphi}(\vc{x}_i)$. Since $\tilde{\varphi}$ is a convex function, due to \eqref{eqn:cvx_var_ineq_subgrad}, one can see that $(\tilde{\thetab},\tilde{\xib})\in \Kcal$. Therefore, we have 
\begin{equation}
\Lcal_{\Kcal,\Dscr}(\tilde\thetab,\tilde\xib) 
= \sum_{i=1}^{\nS}  \|\vc{y}_i + \tilde{\xi}_i\|^2 
= \Lcal_{\Phi,\Dscr}(\tilde{\varphi},\tilde{\xib}) 
<
\Lcal_{\Phi,\Dscr}(\varphi,\xib) =  \Lcal_{\Kcal,\Dscr}(\thetab,\xib),
\end{equation} 
which is a contradiction. This shows that $(\varphi,\xib)$ is a solution of \eqref{eqn:opt_phi}.
This concludes the proof of part {\rm{(ii)}}.
\end{proof}
\begin{theorem}\label{thm:correspondence}
Optimization problem \eqref{eqn:opt_phi} admits a solution in form of 
\begin{equation}\label{eqn:thm_varphi_max}
\varphi(\vc{x}) : = \maxOp_{i=1,\ldots,\nS}\inner{\xi_i}{\vc{x}-\vc{x}_i}+\theta_i,
\end{equation} 
where $(\thetab,\xib)$ is a solution of \eqref{eqn:opt_thetaxi}. Moreover, we have
\begin{equation}
\Lcal_{\Phi,\Dscr}(\varphi,\xib) = \Lcal_{\Kcal,\Dscr}(\thetab,\xib),
\end{equation} 
and $\xi_i\in \partial\varphi(\vc{x}_i)$, for all $i=1,\ldots,\nS$. 
\end{theorem}
\begin{proof}
	Optimization problem \eqref{eqn:opt_thetaxi} is a quadratic program. Since $\Kcal$ is a non-empty polyhedral cone, \eqref{eqn:opt_thetaxi} has a solution, denoted by $(\thetab,\xib)$. Therefore, due to Theorem \ref{thm:correspondence_0}, optimization problem \eqref{eqn:opt_phi} admits a solution in form of \eqref{eqn:thm_varphi_max}.
	The rest of theorem concludes directly from Theorem \ref{thm:correspondence_0} and the given proof.
\end{proof}	
Based on Theorem \ref{thm:correspondence}, one can solve \eqref{eqn:opt_thetaxi} instead of the main estimation problem \eqref{eqn:opt_phi} and  
introduce an estimation of the energy function as in  \eqref{eqn:thm_varphi_max}
where $(\thetab,\xib)$ is a solution of \eqref{eqn:opt_thetaxi}.
However, there are two issues to be addressed: optimization problem \eqref{eqn:opt_thetaxi} does not have a unique solution, and it is not smooth. In the following, we will address these issues.
\subsection{Uniqueness by Regularization}
One can introduce a regularized version of the optimization problem \eqref{eqn:opt_thetaxi} as
\begin{equation}\label{eqn:opt_thetaxi_reg}
\minOp_{(\thetab,\xib)\in \Rbb^{\nS} \times \Rbb^{\nS n}}   \Lcal_{\Kcal,\Dscr}(\thetab,\xib) + \lambda\ \Rcal(\thetab,\xib), 
\end{equation}
where $\Rcal: \Rbb^{\nS} \times \Rbb^{\nS n}\to \Rbbp$ is the regularization function and $\lambda\ge 0$ is the weight of regularization. Based on the next theorem, we introduce a suitable candidate for the regularization function.
\begin{theorem}\label{thm:Zcal_theta_xi}
There exist a unique $\xib^*\in\Rbb^{\nS n}$ and a closed and convex set $\Theta^*\subset\Rbb^n$ such that
\begin{equation}\label{eqn:Zcal_theta_xi}
\Zcal:=\argminOp_{(\thetab,\xib)\in \Rbb^{\nS} \times \Rbb^{\nS n}}   \Lcal_{\Kcal,\Dscr}(\thetab,\xib) = \Theta^*\times\{\xib^*\}.
\end{equation}
\end{theorem}
\begin{proof}
We know that $\Zcal$ is a non-empty closed set. Take $(\thetab_1,\xib_1),(\thetab_2,\xib_2) \in \Zcal \subset\Kcal$.
Since $\Kcal$ is a convex polyhedral cone, one has 
\begin{equation*}
(\thetab,\xib) := \frac{1}{2}(\thetab_1,\xib_1)+\frac{1}{2}(\thetab_2,\xib_2)\in\Kcal.
\end{equation*} 
Moreover, since $\Jcal(\xib):= \sum_{1\le i\le\nS}\|\vc{y}_i+\xi_i\|^2$ is a strongly convex function with $\nabla^2\!\Jcal =2\eye$, we have 
\begin{equation*}
\frac{1}{2}\Jcal(\xib_1) +\frac{1}{2} \Jcal(\xib_2)= \Jcal(\xib) + \frac{1}{4}\|\xib_1-\xib_2\|^2.
\end{equation*}
Accordingly, we should have $\xib_1=\xib_2$, otherwise the problem admits a solution with smaller cost. 
From this argument, we know that there exists a set $\Theta^*\subset \Rbb^n$ such that $\Zcal=\Theta^*\times\{\xib^*\}$. Since $\Zcal$ is a closed set, $\Theta^*$ is also closed. The convexity of $\Theta^*$ follows from the convexity of $\Kcal$ and the fact that the cost function does not depend on $\thetab$.
\end{proof}
Theorem \ref{thm:Zcal_theta_xi} says that the potential non-uniqueness of the solution is due to the term $\thetab$. Accordingly, we consider the regularized cost function $\Jcal_{\lambda}:\Rbb^n\times\Rbb^{\nS n}\to \Rbb$ defined as
\begin{equation}\label{eqn:Jcal_lambda}
\Jcal_{\lambda}(\thetab,\xib) 
:= 
\sum_{i=1}^{\nS}\|\vc{y}_i+\xi_i\|^2 + \lambda \|\thetab\|^2,
\qquad \forall(\thetab,\xib) \in \Rbb^n\times\Rbb^{\nS n}.
\end{equation}  
The next theorem characterizes the solution of the corresponding regularized optimization.
\begin{theorem}\label{thm:Jlambda_unique}
For any $\lambda>0$, the optimization problem 
\begin{equation}\label{eqn:Jlambda}
\minOp_{(\thetab,\xib) \in \Kcal} \Jcal_{\lambda}(\thetab,\xib),
\end{equation}
has a unique solution, denoted by $(\thetabl,\xibl)$.
Moreover, $\lim_{\lambda\to 0}(\thetabl,\xibl)$ exists and is equal to $(\thetab^*,\xib^*)$ where $\thetab^* := \argmin_{\thetab\in\Theta^*}\|\thetab\|^2$.
Also, $\lim_{\lambda\to \infty}(\thetabl,\xibl)$ exists and equals to $(\zero,\xib^\dagger)$ where $ \xib^\dagger$ is the unique solution of  $\min_{(\zero,\xib)\in\Kcal}\sum_{1\le i\le\nS}\|\vc{y}_i+\xi_i\|^2 $.
\end{theorem}
\begin{proof}
One can easily see that $(\zero,\zero)\in\Kcal$ and $\nabla^2\!\Jcal_{\lambda}\succeq \lambda\eye$.
Therefore, \eqref{eqn:Jlambda} is an optimization problem with a strongly convex cost function and non-empty closed and convex feasible set. Therefore, \eqref{eqn:Jlambda} has a unique solution. Similarly, since $\Theta^*$ is non-empty, closed and convex,
$\thetab^* := \argmin_{\thetab\in\Theta^*}\|\thetab\|^2$ is well-defined and exists uniquely.
From the definition of $(\thetab^*,\xib^*)$ and  $(\thetabl,\xibl)$, one can easily see that for any $\lambda>0$, we have 
\begin{equation*}
\begin{array}{rcl}
\sumOp_{i=1}^\nS\|\vc{y}_i+\xi_{i}^*\|^2 
+ 
\lambda\|\thetabl\|^2
&\le &
\sumOp_{i=1}^\nS\|\vc{y}_i+\xi_{\lambda,i}\|^2 
+
\lambda\|\thetabl\|^2
\\ 
&\le&
\sumOp_{i=1}^\nS\|\vc{y}_i+\xi_{i}^*\|^2 
+
\lambda\|\thetab^*\|^2,
\end{array}
\end{equation*}
and subsequently, it holds that $\|\thetabl\|\le\|\thetab^*\|$.
Similarly, since $(\zero,\zero)\in\Kcal$, one can  see that $\|\xibl\|^2\le 4 \sum_{i=1}^\nS\|\vc{y}_i\|^2$
and
$\|\xib^*\|^2\le 4 \sum_{i=1}^\nS\|\vc{y}_i\|^2$.
Now, define set $\Ccal\subset\Rbb^{\nS}\times\Rbb^{\nS n}$ as
\begin{equation}
\Ccal:=\bigg{\{}(\thetab,\xib)\ \bigg{|}\ \|\thetab\|\le\|\thetab^*\|,
\
\|\xib\|\le 2\Big{(}\sum_{i=1}^\nS\|\vc{y}_i\|^2\Big{)}^{\frac{1}{2}}\bigg{\}},
\end{equation}
which is a compact and convex set.
Define $\Zcal_0$ and $\Zcal_{\lambda}$ as
\begin{equation}
\Zcal_0:=\argminOp_{(\thetab,\xib)\in\Kcal\cap\Ccal}\Jcal(\thetab,\xib),
\end{equation}
and
\begin{equation}
\Zcal_{\lambda}:=\argminOp_{(\thetab,\xib)\in\Kcal\cap\Ccal}\Jcal_{\lambda}(\thetab,\xib),
\end{equation}
respectively.
We know that $\Kcal\cap\Ccal$ is a compact set and $\Jcal_{\lambda}(\thetab,\xib)$ is a continuous function with respect to $(\thetab,\xib,\lambda)$. Therefore, due to {\em Maximum Theorem} \cite{Infinitedimensionalanalysis2006}, we know that the set-valued map $\lambda\mapsto\Zcal_{\lambda}$ is upper hemicontinuous with non-empty and compact values. Moreover, one has $\Zcal_0=\{(\thetab^*,\xib^*)\}$ and $\Zcal_{\lambda}=\{(\thetabl,\xibl)\}$. Subsequently, from the upper hemicontinuity of the map $\lambda\mapsto\Zcal_{\lambda}$, we have $\lim_{\lambda\to 0}(\thetabl,\xibl) = (\thetab^*,\xib^*)$. 
Replacing $\lambda$ with $\lambda^{-1}$ and repeating same steps of the proof, one can show the last part of the theorem similarly. 
\end{proof}	
Given $\lambda>0$, we can define our estimator as following
\begin{equation}\label{eqn:varphi_max_lambda}
\varphil(\vc{x}) : = \maxOp_{i=1,\ldots,\nS}\inner{\xi_{\lambda,i}}{\vc{x}-\vc{x}_i}+\theta_{\lambda,i},
\end{equation}
where $(\thetabl,\xibl)$ is the unique solution of \eqref{eqn:Jlambda}.
\begin{remark}
In addition to inducing the uniqueness of solution, the regularization improves the numerical stability 
and the robustness  with respect to noise. 
\end{remark}
\begin{remark}
If further regularization is required for improving the performance of the estimation, we can use a Tikhonov regularization \cite{tikhonov1977solutions} by defining 
$\Rcal(\thetab,\xib) :=\|(\thetab,\xib)\|^2$.
In this case, $\xib$ is also regularized and pushed towards the origin which might be not desirable.
\end{remark}
\subsection{Smoothing the Estimator} 
The smooth version of \eqref{eqn:varphi_max_lambda} is defined as
\begin{equation}\label{eqn:varphi_max_lambda_tau}
\varphilt(\vc{x})
=
\tau \ln
\bigg{(}
\frac{1}{\nS}
\sum_{1\le i\le\nS}
\exp\big{(}\frac{1}{\tau}[\inner{\xi_{\lambda,i}}{\vc{x}-\vc{x}_i}
+
\theta_{\lambda,i}]
\big{)}
\bigg{)},
\end{equation}
where $\tau$ is a positive real scalar \cite{mazumder2019computational}.
\begin{theorem}[\cite{nesterov2018lectures}]
Let the {\em log-sum-exp} function $\ell:\Rbb^n\to\Rbb$ be defined as 
\begin{equation}
\ell(\vc{x}) := \ln(\sum_{i=1}^{n}e^{x_i}),\quad \forall \vc{x}=(x_1,\ldots,x_n)\in\Rbb^n.
\end{equation}	
This function is an analytical and convex function.
The gradient and Hessian of $\ell$ are  
\begin{equation}
\nabla\! \ell(\vc{x}) = \frac{1}{\one^\tr \vc{z}}\vc{z},
\end{equation}
and
\begin{equation}
\nabla^2\! \ell(\vc{x}) = 
\frac{1}{\one^\tr \vc{z}}\Diag(\vc{z}) - 
\frac{1}{(\one^\tr\vc{z})^2}\vc{z}\vc{z}^\tr\,,
\end{equation}
where $ \vc{z} = [e^{x_1},\ldots,e^{x_n}]^\tr$. Moreover, the following holds
\begin{equation}
\begin{array}{c}
\max_{1\le i\le n} x_i \le \ell(\vc{x}) \le 
 \max_{1\le i\le n} x_i + \ln n, \quad  \forall\vc{x}\in\Rbb^n.
\end{array}
\end{equation}
\end{theorem}
This function is used to define a smooth approximation to $\varphil$ in \eqref{eqn:varphi_max_lambda}.
More precisely, let $\Xi_{\lambda}$ be the matrix defined as 
\begin{equation}
\Xi_{\lambda}:=
\begin{bmatrix}
\xi_{\lambda,1} &\xi_{\lambda,2} & \ldots & \xi_{\lambda,\nS}
\end{bmatrix} \in \Rbb^{n\times\nS},
\end{equation}
and, for $i=1,\ldots,\nS$, define $\eta_{\lambda,i}$ as $\eta_{\lambda,i}:= \theta_i-\inner{\xi_{\lambda,i}}{\vc{x}_i}$ and subsequently, let $\etabl$ be the vector defined as 
\begin{equation*}
\etabl := \begin{bmatrix}
\eta_{\lambda,1}&\ldots&\eta_{\lambda,\nS}
\end{bmatrix}^\tr.
\end{equation*}
Subsequently, one can see that
\begin{equation}
\varphilt(\vc{x}) = \tau\ell\Big{(}\frac{1}{\tau}\big{(}\Xi_{\lambda}^\tr\vc{x}+\etabl\big{)}\Big{)}
-
\tau\ln \nS.
\end{equation}
The next corollary motivates the use of $\varphilt$ as a smooth approximant to $\varphil$.
\begin{corollary}\label{cor:log-sum-exp-extended-to-smooth}
For any $\tau>0$, the function $\varphilt$, defined in \eqref{eqn:varphi_max_lambda_tau}, is a convex and analytical function. Moreover, we have	
\begin{equation}
\nabla \varphilt(\vc{x}) = \frac{\Xi_{\lambda}\vc{z}}{\one^\tr \vc{z}},
\end{equation}
and
\begin{equation}
\nabla^2 \varphilt(\vc{x}) = \frac{1}{\one^\tr \vc{z}} \Xi_{\lambda}\diag(\vc{z})\Xi_{\lambda}^\tr - \frac{1}{(\one^\tr\vc{z})^2} \Xi_{\lambda}\vc{z}\vc{z}^\tr\Xi_{\lambda}^\tr,
\end{equation}
where
\begin{equation}
\vc{z} = 
\begin{bmatrix}
\exp\big{(}\frac{1}{\tau}[\inner{\xi_{\lambda,1}}{\vc{x}-\vc{x}_1}
+
\theta_{\lambda,1}]\big{)}
\\\vdots\\
\exp\big{(}\frac{1}{\tau}[\inner{\xi_{\lambda,\nS}}{\vc{x}-\vc{x}_{\nS}}
+
\theta_{\lambda,\nS}]\big{)}
\end{bmatrix}.
\end{equation}
Also, we have the following inequality
\begin{equation}\label{eqn:philt_tight}
\varphil(\vc{x})-
\tau\ln \nS \le \varphilt(\vc{x}) \le 
\varphil(\vc{x}), \quad \forall\vc{x}\in\Rbb^n.
\end{equation}
\end{corollary}
\begin{corollary}\label{cor:tightness-smooth}
By taking $\tau$ small enough, $\varphilt$ in a uniform approximant of $\varphil$. 
More precisely, let $\epsilon$ be an arbitrary positive real scalar and let $\tau < \frac{\epsilon}{\ln \nS}$. Then \eqref{eqn:philt_tight} shows that $|\varphil(\vc{x})-
\varphilt(\vc{x})|<\epsilon$, for any $\vc{x}\in\Rbb^n$.
\end{corollary}
\subsection{Further Extensions}\label{ssec:extentions}
We now introduce  extensions of the proposed estimation strategy to  other settings. 
We briefly present how the main estimation problem \eqref{eqn:opt_phi} and its finite-dimensional version \eqref{eqn:opt_thetaxi} are adapted. The regularization and the smoothing procedures follow the same lines as before.
\subsubsection{Concave Energy Functions} We know that $\psi$ is a concave function if and only if $\varphi:=-\psi$ is a convex function. Accordingly, the estimation problem \eqref{eqn:opt_phi} is modified to the following optimization problem
\begin{equation}\label{eqn:opt_phi_concave}
\begin{array}{cl}
\minOp_{
	\xi_1,\ldots,\xi_{\nS}\in \Rbb^n,\ \varphi\in\Phi} &  \sumOp_{i=1}^{\nS} \|\vc{y}_i-\xi_i\|^2,\\
\textrm{s.t.} & \xi_i\in\partial\varphi(\vc{x}_i), \quad\forall i=1,\ldots,\nS. 
\end{array}
\end{equation} 
Analogously to optimization problem \eqref{eqn:opt_thetaxi}, 
this leads to the following finite dimensional problem
\begin{equation}\label{eqn:opt_thetaxi_concave}\!\!
\begin{array}{cl}
\minOp_{\substack{
		\xi_1,\ldots,\xi_{\nS}\in \Rbb^n\!\!\\ \!\!\!\theta_1,\ldots,\theta_{\nS} \in\Rbb\!\!}} &  \sumOp_{i=1}^{\nS} \|\vc{y}_i-\xi_i\|^2,
\\
\textrm{s.t.} & \theta_j-\theta_i\ge \inner{\xi_i}{\vc{x}_j-\vc{x}_i}, \quad \forall  i,j\in\{1,\ldots,\nS\}. 
\end{array}
\end{equation}
\subsubsection{Strongly Convex Energy Functions} 
For $\mu>0$, the convex function $\varphi$ is said to be $\mu$-strongly convex if for any $\vc{x},\vc{y}\in\Rbb^n$, we have
\begin{equation*}\label{eqn:cvx_var_ineq_subgrad_strongly_cvx}
\varphi(\vc{y})-\varphi(\vc{x})\ge \inner{\xi}{\vc{y}-\vc{x}}+\frac{\mu}{2}\|\vc{x}-\vc{y}\|^2,\quad \forall \xi\in\partial\varphi(\vc{x}).
\end{equation*}
Let $\Phi_\mu$ denote the set of $\mu$-strongly convex functions. If we know that the energy function $\varphi$ belongs to $\Phi_\mu$, then the estimation problem 
\eqref{eqn:opt_phi} is adapted to the following optimization problem
\begin{equation}\label{eqn:opt_phi_strongly_cvx}
\begin{array}{cl}
\minOp_{
	\xi_1,\ldots,\xi_{\nS}\in \Rbb^n,\ \! \varphi\in\Phi_\mu} &  \sumOp_{i=1}^{\nS} \|\vc{y}_i+\xi_i\|^2,\\
\textrm{s.t.} & \xi_i\in\partial\varphi(\vc{x}_i), \quad \forall i=1,\ldots,\nS. 
\end{array}
\end{equation} 
In this case, the optimization problem \eqref{eqn:opt_thetaxi} is modified to \begin{equation}\label{eqn:opt_thetaxi_strongly_cvx}
\begin{array}{cl}
\minOp_{\substack{
		\xi_1,\ldots,\xi_{\nS}\in \Rbb^n\!\!\\ \!\!\!\theta_1,\ldots,\theta_{\nS} \in\Rbb\!\!}} &  \sumOp_{i=1}^{\nS} \|\vc{y}_i+\xi_i\|^2,
\\
\textrm{s.t.} & \theta_j-\theta_i\ge \inner{\xi_i}{\vc{x}_j-\vc{x}_i} +\frac{\mu}{2}\|\vc{x}_i-\vc{x}_j\|^2,\quad \forall
i,j\in\{1,\ldots,\nS\}. 
\end{array}
\end{equation}
In the case of $\mu$-strongly concave energy functions, this can be adapted along the lines of the formulation in \eqref{eqn:opt_thetaxi_concave} .
\subsubsection{Including Knowledge of Equilibrium Points} 
Let assume that we know $\vc{x}_0 = \zero$ is an equilibrium of the dynamical system. Accordingly, we need to have $\zero\in \partial\varphi(\zero)$. Therefore, in order to incorporate this knowledge, the estimation problem 
\eqref{eqn:opt_phi} should be modified to the following,
\begin{equation}\label{eqn:opt_phi_equilibrium}
\begin{array}{cl}
\minOp_{
	\xi_1,\ldots,\xi_{\nS}\in \Rbb^n,\ \! \varphi\in\Phi} 
&  
\sumOp_{i=1}^{\nS} \|\vc{y}_i+\xi_i\|^2,\\
\textrm{s.t.} 
&  
\xi_i\in\partial\varphi(\vc{x}_i), \quad i=1,\ldots,\nS,
\\&
\zero\in \partial\varphi(\zero).
\end{array}
\end{equation} 
Without loss of generality, we can assume that $\varphi(\zero)=0$. 
Accordingly, one can set $\theta_0=0$ and $\xi_0=\zero$.
Therefore, we modify the optimization problem \eqref{eqn:opt_thetaxi} as following 
\begin{equation}\label{eqn:opt_thetaxi_equilibrium}
\begin{array}{cl}
\minOp_{\substack{
		\xi_0,\ldots,\xi_{\nS}\in \Rbb^n\!\!\\ \!\!\!\theta_0,\ldots,\theta_{\nS} \in\Rbb\!\!}} &  \sumOp_{i=1}^{\nS} \|\vc{y}_i+\xi_i\|^2,
\\
\textrm{s.t.} & \theta_j-\theta_i\ge \inner{\xi_i}{\vc{x}_j-\vc{x}_i}, \quad \forall  i,j\in\{0,1,\ldots,\nS\},
\\& \theta_0=0,
\\& \xi_0=\zero.
\end{array}
\end{equation}

%% file: sec_05_NonCVX_Energy.tex
\section{General Energy Functions}\label{sec:noncvx_energy}
In this section, we consider general energy functions.
For the sake of generality, the differentiability assumption of the energy function is relaxed initially.
The next theorem plays a key role in the formulation of the estimation problem.
\begin{theorem}[\cite{yuille2003concave}]\label{thm:diff_cvx} {\rm (i)}
Let $\Omega\subseteq \Rbb^n$ be a  convex set and $\varphi:\Omega\to \Rbb$ be a $C^2(\Omega,\Rbb)$ function with bounded Hessian, i.e., $\sup_{\vc{x}\in\Omega}\|\nabla^2\!\varphi(\vc{x})\|<\infty$. Then, there exist convex functions $\varphione,\varphitwo:\Omega\to \Rbb$ such that $\varphi(\vc{x}) = \varphione(\vc{x})-\varphitwo(\vc{x})$, for any $\vc{x}\in\Omega$.
{\rm (ii)} Moreover, if $\Omega$ is convex and compact, then the Hessian is bounded and $\varphi$ is decomposable as the difference of two convex functions. 
\end{theorem}
We construct the {\em loss function} for the estimation problem, denoted by $\Lscr_{\Phi,\Dscr}$, as the sum of squared errors. 
More precisely, the function $\Lscr_{\Phi,\Dscr}:\Phi\times\Phi\times\Rbb^{\nS n}\times\Rbb^{\nS n}\to \Rbb$ is defined as 
\begin{equation}\label{eqn:loss_phi_DC}
\Lscr_{\Phi,\Dscr}(\varphione,\varphitwo,\xibone,\xibtwo) 
:=
\sum_{i=1}^{\nS} 
\|\vc{y}_i
+
\xione_i
-
\xitwo_i\|^2
+ 
\sum_{i=1}^{\nS} \Ical_{ \partial \varphione(\vc{x}_i)}(\xione_i) 
+
\sum_{i=1}^{\nS} \Ical_{ \partial \varphitwo(\vc{x}_i)}(\xitwo_i), 
\end{equation}
for any pair of convex functions $\varphione,\varphitwo \in \Phi$ and vectors $\xione_1,\ldots,\xione_{\nS},\xitwo_1,\ldots,\xitwo_{\nS}\in\Rbb^n$, 
$\xibone,\xibtwo \in\Rbb^{\nS n}$ where $\xibone,\xibtwo \in\Rbb^{\nS n}$ are column vectors respectively defined as $\xibone :=
[\xione_1{}^\tr  \ldots\  \xione_{\nS}{}^\tr ]^\tr$
and
$\xibtwo :=
[\xitwo_1{}^\tr  \ldots\  \xitwo_{\nS}{}^\tr ]^\tr$.
Accordingly, the estimation problem is defined as
\begin{equation}\label{eqn:opt_phi_DC}
\begin{array}{cl}
\minOp_{
	\substack{
\xione_1,\ldots,\xione_{\nS}\in \Rbb^n, \varphione\in\Phi\\
\xitwo_1,\ldots,\xitwo_{\nS}\in \Rbb^n, \varphitwo\in\Phi}
	} 
&
\sumOp_{i=1}^{\nS} \|\vc{y}_i+\xione_i-\xitwo_i\|^2,\\
\textrm{s.t.} 
&
\xione_i \in \partial\varphione(\vc{x}_i), \quad \forall  i=1,\ldots,\nS,
\\
&
\xitwo_i \in \partial\varphitwo(\vc{x}_i), \quad \forall i=1,\ldots,\nS.  
\end{array}
\end{equation}
Analogous to the previous section, we can introduce a finite-dimensional formulation as following
\begin{equation}\label{eqn:opt_thetaxi_DC}
\begin{array}{cl}
\minOp_{\substack{
		\thetabone,\thetabtwo\in\Rbb^n\\ \xibone,\xibtwo\in \Rbb^{\nS n}}} 
&
\sumOp_{i=1}^{\nS} \|\vc{y}_i+\xione_i-\xitwo_i\|^2,
\\
\textrm{s.t.} 
& 
\thetaone_j-\thetaone_i \ge  \inner{\xione_i}{\vc{x}_j - \vc{x}_i}, \quad\forall i,j\in\{1,\ldots,\nS\},
\\
& 
\thetatwo_j-\thetatwo_i \ge  \inner{\xitwo_i}{\vc{x}_j - \vc{x}_i}, \quad\forall i,j\in\{1,\ldots,\nS\},  
\end{array}
\end{equation}
where $\thetabone,\thetabtwo\in\Rbb^n$ are defined respectively as $\thetabone =[\thetaone_1,\ldots,\thetaone_{\nS}]^\tr$ and
$\thetabtwo =[\thetatwo_1,\ldots,\thetatwo_{\nS}]^\tr$.
Considering this optimization problem, we define a loss function $\Lscr_{\Kcal,\Dscr}:\Rbb^{\nS} \times\Rbb^{\nS} \times \Rbb^{\nS n}\times\Rbb^{\nS n}\to \bar{\Rbb}$  as
\begin{equation}
\Lscr_{\Kcal,\Dscr}(\thetabone,\thetabtwo,\xibone,\xibtwo) := \sumOp_{i=1}^{\nS} \|\vc{y}_i+\xione_i-\xitwo_i\|^2 
+ \Ical_{\Kcal}(\thetabone,\xibone)+ \Ical_{\Kcal}(\thetabtwo,\xibtwo).
\end{equation} 
With lines of proof similar to those in Section \ref{sec:cvx_energy}, we formalize the connection between optimization problems \eqref{eqn:opt_phi_DC} and \eqref{eqn:opt_thetaxi_DC}.
\begin{theorem}\label{thm:correspondence_0_DC}
{\rm (i)} Let $(\varphione\!,\varphitwo\!,\xibone\!,\xibtwo)$ be a solution of \eqref{eqn:opt_phi_DC}. Then, \eqref{eqn:opt_thetaxi_DC} has a solution $(\thetabone\!,\thetabtwo\!,\xibone\!,\xibtwo)$ such that 
\begin{equation}
\Lscr_{\Kcal,\Dscr}(\thetabone,\thetabtwo,\xibone,\xibtwo)
=
\Lscr_{\Phi,\Dscr}(\varphione,\varphitwo,\xibone,\xibtwo).
\end{equation}
{\rm (ii)} If $(\thetabone,\thetabtwo,\xibone,\xibtwo)$ is a solution of \eqref{eqn:opt_thetaxi_DC}, then \eqref{eqn:opt_phi_DC} has a solution $(\varphione,\varphitwo,\xibone,\xibtwo)$ such that 
\begin{equation}
\Lscr_{\Phi,\Dscr}(\varphione,\varphitwo,\xibone,\xibtwo) = \Lscr_{\Kcal,\Dscr}(\thetabone,\thetabtwo,\xibone,\xibtwo),
\end{equation}
and $\xione_i\in \partial\varphione(\vc{x}_i)$, $\xitwo_i\in \partial\varphitwo(\vc{x}_i)$, for all $i=1,\ldots,\nS$.
\end{theorem}
\begin{theorem}\label{thm:correspondence_DC}
Optimization problem \eqref{eqn:opt_phi_DC} admits a solution in form of 
\begin{equation}\label{eqn:varphi_max_DC}
\varphil(\vc{x}) : =  \varphione(\vc{x}) - \varphitwo(\vc{x}) 
=
\max_{1\le i_1\le\nS}\inner{\xione_{i_1}}{\vc{x}-\vc{x}_{i_1}} 
+
\thetaone_{i_1}
-
\max_{1\le i_2\le\nS}\inner{\xitwo_{i_2}}{\vc{x}-\vc{x}_{i_2}}
+
\thetatwo_{i_2}.
\end{equation}
where $(\thetabone,\thetabtwo,\xibone,\xibtwo)$ is a solution of \eqref{eqn:opt_thetaxi_DC}. Moreover, we have
\begin{equation}
\Lscr_{\Phi,\Dscr}(\varphione,\varphitwo,\xibone,\xibtwo) = \Lscr_{\Kcal,\Dscr}(\thetabone,\thetabtwo,\xibone,\xibtwo),
\end{equation}
and $\xione_i\in \partial\varphione(\vc{x}_i)$, $\xitwo_i\in \partial\varphitwo(\vc{x}_i)$, for all $i=1,\ldots,\nS$.
\end{theorem}
As in Section \ref{sec:cvx_energy}, regularization can be used for imposing uniqueness in the estimation.
Using the same arguments as those given in the proof of Theorem \ref{thm:Zcal_theta_xi}, one can obtain a similar conclusion and subsequently  show that the difference $\xibone-\xibtwo $ is unique and the potential non-uniqueness of the solution is due to the other terms. Consequently, one can introduce the regularized cost function $\Jcal_{\lambda}:\Rbb^n
\!\times\!\Rbb^n\!\times\!\Rbb^{\nS n}\!\times\!\Rbb^{\nS n}\!\to\! \Rbb$ as 
\begin{equation}\label{eqn:Jcal_lambda_DC}
\Jcal_{\lambda}(\thetabone,\thetabtwo,\xibone,\xibtwo) := \sum_{i=1}^{\nS}\|\vc{y}_i+\xione_i-\xitwo_i\|^2  
+
\lambda \Big{(}\|\thetabone\|^2+\|\thetabtwo\|^2+\|\xibone+\xibtwo\|^2\Big{)},
\end{equation}  
and solve the following regularized optimization problem
\begin{equation}\label{eqn:Jlambda_DC}
\minOp_{(\thetabone,\xibone)\in \Kcal,(\thetabtwo,\xibtwo)\in \Kcal} \Jcal_{\lambda}(\thetabone,\thetabtwo,\xibone,\xibtwo),
\end{equation}
where $\lambda> 0$ is the regularization weight. Similar to Theorem \ref{thm:Jlambda_unique}, one can show that, for any $\lambda>0$, the regularized optimization problem \eqref{eqn:Jlambda_DC} has a unique solution, denoted by  $(\thetablone\!,\thetabltwo\!,\xiblone\!,\xibltwo)$.
Consequently, we define $\varphilone$ and $\varphiltwo$ similar to \eqref{eqn:varphi_max_lambda}, and thus, 
the estimation is defined as
\begin{equation*}\label{eqn:varphi_max_lambda_DC}
\varphil(\vc{x}) : =  \varphilone(\vc{x}) - \varphiltwo(\vc{x}) 
=
\max_{1\le i_1\le\nS}\inner{\xione_{\lambda,i_1}}{\vc{x}-\vc{x}_{i_1}} 
+
\thetaone_{\lambda,i_1}
-
\max_{1\le i_2\le\nS}\inner{\xitwo_{\lambda,i_2}}{\vc{x}-\vc{x}_{i_2}}
+
\thetatwo_{\lambda,i_2}.
\end{equation*}
Similar to the previous section, we can smooth this estimator 
using log-sum-exp function. 
In this regard, let 
$\varphiltone$ and $\varphilttwo$ be defined as in \eqref{eqn:varphi_max_lambda_tau}, and then define the {\em smooth estimator}, denoted by $\varphilt$, as 
\begin{equation}\label{eqn:phi_dc_smoothed}
\begin{array}{c}
\varphilt(\vc{x}) = \varphiltone(\vc{x})-\varphilttwo(\vc{x}).
\end{array}
\end{equation}
Note that Corollary \ref{cor:log-sum-exp-extended-to-smooth} and Corollary \ref{cor:tightness-smooth} are valid for $\varphiltone$ and $\varphilttwo$. 
Moreover, we have the following corollary for $\varphilt$. 
\begin{corollary}\label{cor:tightness-smooth_DC}
Let $\epsilon$ be an arbitrary positive real scalar and let $\tau < \frac{\epsilon}{2\ln \nS}$. Then, due to  \eqref{eqn:philt_tight},  we have $|\varphil(\vc{x})-
\varphilt(\vc{x})|<\epsilon$, for any $\vc{x}\in\Rbb^n$.
In other words, we can uniformly approximate $\varphil$ with $\varphilt$ to an arbitrary accuracy by taking $\tau$ sufficiently small.
\end{corollary}
\begin{remark}
Let $\{h_k\}_{k=1}^p$ be a set of given vector fields. 
Then, the proposed method can be extended to the case where the dynamics is sum of a gradient flow and a parametric part, i.e., 
\begin{equation}
f(\vc{x})=-\nabla\! \varphi(\vc{x})+\sum_{k=1}^p \alpha_k h_k(\vc{x}).
\end{equation} 
To estimate $f$, we modify optimization problem \eqref{eqn:opt_phi_DC} to 
\begin{equation}\label{eqn:opt_phi_DC_para}
\begin{array}{cl}
\minOp_{
	\substack{
		\xibone\in \Rbb^{\nS n},\ \! \varphione\in\Phi\\
		\xibtwo\in \Rbb^{\nS n},\ \! \varphitwo\in\Phi\\
		\alpha_1,\ldots,\alpha_p\in\Rbb}
} 
& 
\sumOp_{i=1}^{\nS} \|\vc{y}_i + \xione_i - \xitwo_i - \sumOp_{k=1}^p \alpha_k h_k(\vc{x}_i)\|^2,
\\\mathrm{s.t.} 
&
\xione_i \in \partial\varphione(\vc{x}_i), \quad \forall  i=1,\ldots,\nS,
\\
&
\xitwo_i \in \partial\varphitwo(\vc{x}_i), \quad \forall i=1,\ldots,\nS.  
\end{array}
\end{equation}
and apply the previous adaptations to make the problem finite dimensional and sufficiently smooth.
\end{remark}	

%% file: sec_06_NumericalResults.tex
\begin{figure}[t]
	\centering
	\includegraphics[trim={2.75cm 5.75cm 3.5cm 6.4cm},clip,width=0.75\textwidth]{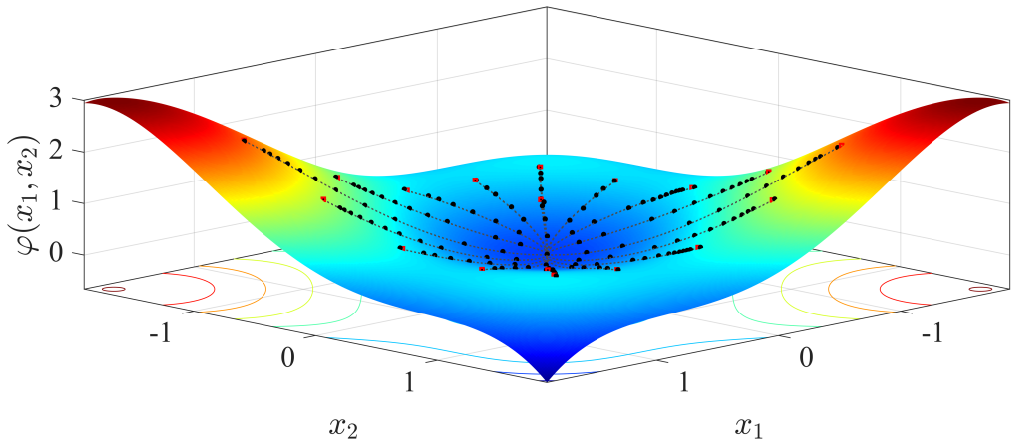}
	\caption{The  graph of the energy function \eqref{eqn:example_phi} and the sampled trajectories. Red bullets show the starting points of trajectories while the black bullets show the sampling points.}
	\label{fig:phi}
\end{figure}
\section{Numerical Experiments}\label{sec:numerics}
In this section, we discuss a numerical example.
To this end, consider the non-convex energy function $\varphi:\Rbb^2\to\Rbb$ defined as
\begin{equation}\label{eqn:example_phi}
\varphi(x_1,x_2) = 
ax_1^2+  
bx_1x_2+  
ax_2^2
-cx_1^4-cx_2^4,  
\end{equation}
where $a = 0.7, b= -0.5 $ and $c = 0.15$.
The graph of energy function is shown in Figure \ref{fig:phi}.
The gradient flow dynamics of \eqref{eqn:example_phi} are given by
\begin{equation}
\begin{array}{r}
\dot{x}_1 = f_1(\vc{x}) =  -\partial_{x_1}\varphi(x_1,x_2) = -2ax_1-bx_2+4cx_1^3, 
\\
\dot{x}_2 = f_2(\vc{x}) =  -\partial_{x_2}\varphi(x_1,x_2) = -bx_1-2ax_2+4cx_2^3. 
\end{array}
\end{equation}
Since the gradient flows are curl-free, a single  trajectory does not explore the space.
Therefore, in order to collect sufficient data points for identifying the dynamics, it is required to take a  set of sufficiently rich initial points and sample the resulting trajectories.
We consider $18$ initial points in 
\begin{equation}
\Omega:=[-1.9,1.9]\times[-1.9,1.9]\subset \Rbb^2.
\end{equation} 
From the resulting trajectories,  
we take $118$ noisy samples with additive white Gaussian noise of zero mean and standard deviation $\sigma_w=0.01$.
In Figure \ref{fig:phi}, the initial points, the sampled points and the trajectories are shown by red bullets, black bullets and dotted lines, respectively.
The derivatives are estimated using MATLAB tools for {\em curve-fitting}.
From these point, $80\%$ are randomly chosen as the data set for estimating $f(\vc{x})=-\nabla\!\varphi(\vc{x})$. 
By solving \eqref{eqn:Jlambda_DC}, a close to minimum norm estimation is obtained (see Theorem \ref{thm:Jlambda_unique} and Section \ref{sec:noncvx_energy}). 
Then, the results are used to construct an estimate of the energy function as in \eqref{eqn:varphi_max_lambda_DC}.
Further more,  using the {\em log-sum-exp} function a smooth version of $\varphi$  is derived as in \eqref{eqn:phi_dc_smoothed}.
The parameters $\lambda$ and $\tau$ are chosen based on a cross-validation procedure performed using  the remaining $20\%$ of the data points. The results are $\lambda=10^{-8}$ and $\tau=0.16$.
Following this, we estimate $\partial_{x_1}\varphi(x_1,x_2)$ and $\partial_{x_2}\varphi(x_1,x_2)$ due to the gradient of the smoothed function and Corollary \ref{cor:log-sum-exp-extended-to-smooth}.
Figure \ref{fig:varphi_varphi_hat} shows that the results closely fit the true values.
The calculated {coefficient of determination} for the estimations, also known as {\em R-squared},  equals to $92.4\%$. 

\begin{figure}[t]
	\centering
	\includegraphics[width=0.85\textwidth]{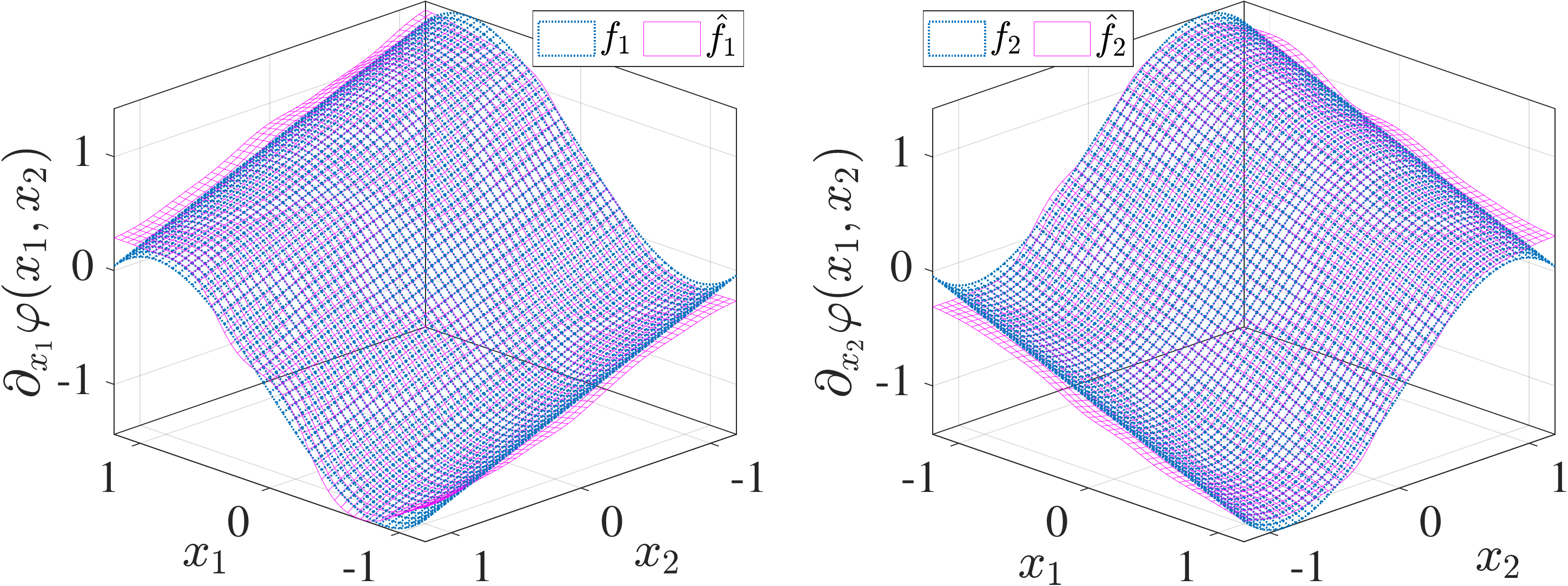}
	\caption{In left, the first coordinate of $\nabla\!\varphi$ and the corresponding estimations are shown in blue and magenta, respectively. Similarly, right shows the second coordinate of $\nabla\!\varphi$ and its estimation respectively in blue and magenta. The plots are rotated for clarity.}
	\label{fig:varphi_varphi_hat}
\end{figure}

%% file: sec_07_Concl.tex
 \section{Conclusion}\label{sec:con}
We have introduced a nonparametric system identification method for nonlinear systems with gradient-flow dynamics. The corresponding vector field is the gradient field of a potential energy function. 
This fact is structural prior knowledge as so is used as a constraint in the proposed method. Initially, the identification problem is formulated as a minimization of the fitting or prediction error over the hypothesis space of convex functions. To give a tractable problem, an equivalent formulation is derived as a finite dimensional quadratic program. This formulation is based on two central ideas: representing the energy function as a difference of two convex functions, and a necessary and sufficient condition for convexity.
The final optimization problem approximates the convex functions jointly. The existence, uniqueness and smoothness of the solution is addressed. Finally, a numerical example is presented where a non-convex energy function is considered and based on sample data from multiple trajectories, the corresponding gradient flow is identified.	